\documentclass[12pt]{amsart}
\usepackage{amsmath,amsfonts,amssymb,amsthm}
\newtheorem {theorem}{Theorem}[section]

\newtheorem {lemma}{Lemma}[section]

\newtheorem {problem}{Problem}

\newtheorem {definition}{Definition}[section]

\def\ar{a\kern-.370em\raise.16ex\hbox{\char95\kern-0.53ex\char'47}\kern.05em}
\def\ees{{\accent"5E e}\kern-.385em\raise.2ex\hbox{\char'23}\kern-.08em}
\def\EES{{\accent"5E E}\kern-.5em\raise.8ex\hbox{\char'23 }}
\def\ow{o\kern-.42em\raise.82ex\hbox{
   \vrule width .12em height .0ex depth .075ex \kern-0.16em \char'56}\kern-.07em}
\def\OW{O\kern-.460em\raise1.36ex\hbox{
\vrule width .13em height .0ex depth .075ex \kern-0.16em \char'56}\kern-.07em}

\title{A remark on the extension of a class of linear functionals}

\author{Ho\`ang Phi D\~ung $\dagger$}

\address{Laboratory of Applied Mathematics and Computing, Posts and Telecommunications Institute of Technology (PTIT), Vietnam}
\address{Faculty of Fundamental Studies, PTIT, Office A2, Fl.10, Km10 Nguyen Trai Rd., Ha Dong District, Hanoi, Vietnam}
\email{dunghp@ptit.edu.vn}

\date{}



\begin{document}

\maketitle
\begin{abstract}
We will remark an extension of a linear functional on subalgebra of algebra of continuous functions on subset of $\mathbb{R}^n$ which preserves positivity.
\end{abstract}
\section{Introduction}
In the works published in the late 1920s, Banach gave a sufficient condition to be able to extend a linear functional from a closed subspace of a Banach space to this Banach space (\cite{Banach}). After that, when studying classical moment problems, M. Riesz gave his general extension theorem, this theorem said that a positive linear functional on the convex cone $K$ (with some assumptions) has an positive extension (\cite{MRiesz}).  Researching on extension of a linear functional is meaningful and the results have many applications to many branches of mathematics such as: functional analysis, partial differential equations, convex analysis, optimization, general topology \dots Among them, class of continuous linear functionals have property: when extending this functional, it preserves positivity. This is an important functional class.

In Haviland's works on the problem of moments (\cite{Ha1, Ha2}), he gave a complete solution of following classical moment problem:
\begin{problem}[Classical K-moment problem]
The sequence $\{s_n\}_{n=0}^\infty (s_0 = 1)$ is a $K$-moment sequence if there exists a finite Borel measure (positive) $\mu$ on $K \subseteq \mathbb{R}$ such that: $$ s_k = \int_K x^k d\mu(x), k = 0,1,2, \dots $$
When does this sequence $\{s_n\}_{n=0}^\infty$ is a $K$-moment sequence?
\end{problem}
These problems have complete answers, but Haviland's solution is more generalized: The sequence $\{s_n\}_{n=0}^\infty (s_0 = 1)$ is a $K$-moment sequence if and only if $L(f) \ge 0$ for all polynomials $f \in \mathbb{R}[x]$ nonnegative on $K$, with $L$ is a positive linear functional on $\mathbb{R}[x]$. Actually, he extended $L$ from $\mathbb{R}[x]$ to larger algebra.

The modern moment problem is considered in the relation with theory of self-adjoint operators (see \cite{Fuglede, Si}), abelian semi-groups (see \cite{BCR, Berg}), positive definite functions (see \cite{BM, Schmudgen}), real algebraic geometry and optimization (see \cite{Schmudgen, Lasserre}). In \cite{Schmudgen}, the author gave a criterion for general K-moment problem in $\mathbb{R}^n$ with $K$ is a semialgebraic set.
Surveys of the classical moment problem and related topics the reader can find in \cite{Akh}, \cite{Sho}, \cite{Landau}.

In 1907, F. Riesz proved his Representation Theorem (\cite{FRiesz}), this theorem asserts that with any positive linear functional, there exists unique Borel measure such that this linear functional is an integral corresponding to above measure. In other words, there is a bijection between all positive linear functionals on $C_c(X)$ (space of continuous functions on $X \subseteq \mathbb{R}^n$ with the compact support) and all Borel measures. 

The problem of positive extension of linear functionals was developed by some authors, such as S. Mazur and W. Orlicz (\cite{MO}), O. Hustad (\cite{Hustad1, Hustad2}), Kantorovich (\cite{Kantorovich}), Krein and Smulian (\cite{Krein}), G. Choquet (\cite{Choq}), \dots 

The authors in \cite{MO} and \cite{Hustad2} used some linear inequalities to handle the problem of extension of positive linear functionals. Kantorovich and Krein-Smulian, independently, gave a sufficient condition to be able to extend a positive linear functional. Note that, these results can be proved by using Hahn-Banach Theorem. Recently, O. Olteanu explicitly pointed out the relation between the extension of positive linear functionals and the moment problems (see \cite{Olteanu1, Olteanu2}). In the past, G. Choquet (\cite{Choq}) gave some results on a bounded positive linear functional on space of continuous functions on a locally compact Hausdorff space $X$ can extend to an adapted space, a generalized object. He also gave a generalized solution of the moment problem.

Recently, M. Marshall improved G. Choquet's theorem (\cite[Theorem 3.1]{Marshall1}), the author also gave a general criterion for the representability of a positive linear functional as an integral.

In this paper, we give another version of Marshall's theorem, this version gave an extension of a positive linear functional from a subalgebra of $C(X)$ to larger subalgebra containing this subalgebra and the algebra of continuous functions with compact support, this extension preserves positivity of the linear functional. We give proof of this theorem by using Hahn-Banach theorem, being similar to Marshall's proof.
\section{Preliminaries}
\begin{definition} Let $X$ be a subset of $\mathbb{R}^n$ and $C(X)$ be algebra of continuous functions on $X$. A positive linear functional on $C(X)$ is a linear functional $L$ with $L(f) \ge 0$ for all $f \in C(X)$ such that $f(a) \ge 0, \forall a \in X$.
\end{definition}
We recall Haviland's result in \cite{Marshall2} (also see \cite{Ha1, Ha2}), with $\mathbb{R}[x_1, \dots, x_n]$ denotes the ring of real multivariable polynomials:
\begin{theorem}[Haviland]
For a linear functional $L: \mathbb{R}[x_1, \dots, x_n]$ and closed set $K$ in $\mathbb{R}^n$, the following are equivalent:
\begin{enumerate}
\item[{\rm(i)}] $L$ comes from a Borel measure on $K$, i.e., $\exists$ a Borel measure $\mu$ on $K$ such that, $\forall f \in \mathbb{R}[x_1, \dots, x_n], L(f) = \int f d\mu.$
\item[{\rm(ii)}] $L(f) \ge 0$ holds for all $f \in \mathbb{R}[x_1, \dots, x_n]$ such that $f \ge 0$ on $K$.
\end{enumerate}
\end{theorem}
In Haviland's theorem, a positive linear functional extended from ring of real multivariable polynomials to larger subalgebra and this theorem can be derived as a consequence of the following the Riesz Representation Theorem (see \cite[p. 77]{KS}):
\begin{theorem}[Riesz Representation Theorem]
Let $X$ be a locally compact Hausdorff space and let $L: C_c(X) \to \mathbb{R}$ be a positive linear functional. Then there exists a unique Borel measure $\mu$ on $X$ such that $$ L(f) = \int f d\mu, \forall f \in C_c(X). $$
$C_c(X)$ is the algebra of continuous functions with compact support. 
\end{theorem}
The following lemma (see \cite{Akh} or \cite{Marshall2}) will be used to estasblish our results:
\begin{lemma}\label{lemma1} Suppose that $f$ is a non-zero polynomial in the one variable $x$ and $f(a) \ge 0, \forall a \in \mathbb{R}$. Then $f(x)$ is a sum of squares of polynomials.
\end{lemma}
\section{Main results}
We give a generalized result of the extension of positive linear functionals. We learnt the method of proof of Marshall (\cite[Theorem 3.1]{Marshall1}) to prove this result
\begin{theorem}\label{thm1} Let $X$ be a locally Hausdorff space and $C(X)$ be a algebra of continuous functions on $X$. Suppose that $\mathcal{A}$ be a subalgebra (with unit 1) of $C(X)$, $X$ is a subset of $\mathbb{R}^n$. Then, for each linear functional $L : \mathcal{A} \to \mathbb{R}$ satisfying $$ L(a) \ge 0, \forall a \in \mathcal{A}\ \text{such that}\ a(x) \ge 0, \forall x \in X,$$ there exists an extension of $L$ on an algebra $C'(X)$ which contains $\mathcal{A} \cup C_c (X)$ and preserves positivity.
\end{theorem}
\begin{proof}
We denote $C'(X)$ set of continuous real-valued function $$ f : X \to \mathbb{R} $$ and be bounded by $a \in \mathcal{A}$, i.e. there exists $a \in \mathcal{A}$ such that $|f(x)| \le |a(x)|, \forall x \in X$.

It is easy to see that $C'(X)$ is an algebra.

We have: $\mathcal{A} \subseteq C'(X) \subseteq C(X)$ and $C'(X) \subseteq C(X)$ is trivial.

Indeed, $\mathcal{A} \subseteq C'(X)$ because we get any $a \in \mathcal{A}$ with $a$ is a continuous function on $X$ and $\left| a\right| \le \left|a\right|$, thus $a \in C'(X)$. Moreover, $C_c(X) \subseteq C'(X)$. Thus $\mathcal{A} \cup C_c(X) \subseteq C'(X)$.

If $f \in C_c(X)$ then $f$ is continuous and has compact support. Hence, $\left|f\right| \le n$ and $n \in\mathcal{A}$ ($n = n\cdot \mathbf{1}, \mathbf{1} : X \to \mathbb{R}, x \mapsto 1$), thus, $f \in C'(X)$.

We will extend functional $L :\mathcal{A} \to \mathbb{R}$ from $\mathcal{A}$ to $C'(X)$ which preserves positive definite. Let $$\mathcal{E} = \{(\bar{L},V) | \mathcal{A} \subseteq V\ \text{and}\ \bar{L} : V \to \mathbb{R}\ \text{satisfies}\ \bar{L}(v) \ge 0, \forall v \ge 0, v \in V\ \text{and}\ \bar{L}_{|_\mathcal{A}} = L\}.$$
We consider ordered relation in $\mathcal{E}$: $(L_1,V_1) \le (L_2,V_2) \Leftrightarrow \begin{cases}
& V_1 \subseteq V_2\\
& {L_2}_{|_{V_1}} = L_1
\end{cases}$\\
Every chains are bounded by upper bound $(L', V')$ with $$ V' = \bigcup_{i=1}V_i \ \ (V_1 = A) $$ $$ L'_{|_{V_i}} = L_i. $$

Using Zorn's Lemma, $\mathcal{E}$ has maximum element and we call it $(\bar{L}, V)$: $$\mathcal{A} \subseteq V \subseteq C'(X).$$
We will prove that $V = C'(X)$. Suppose that there exists $g \in C'(X)$ satisfying $g \notin V$.\\
Because $g \in C'(X)$, $\exists a \in \mathcal{A}$ such that: $|g| \le |a|$ and $|a| \le \dfrac{a^2 + 1}{2}$ (by $(a \pm 1)^2 \ge 0$), therefore, $|g| \le \dfrac{a^2 + 1}{2}$ or $-\dfrac{a^2 + 1}{2} \le g \le \dfrac{a^2 + 1}{2}$.\\
Let $f_1 = - \dfrac{a^2 + 1}{2}, f_2 = \dfrac{a^2 + 1}{2}$, then $f_1 \le g \le f_2$. By $f_1, f_2 \in \mathcal{A}$, so $f_1, f_2 \in V$ with $f_1 \le f_2$, this implies $\bar{L}(f_1) \le \bar{L}(f_2)$. From that, let we consider $\sup\{\bar{L}(f_1) | f_1 \in V, f_1 \le g\}$ and $\inf\{\bar{L}(f_2) | f_2 \in V, g \le f_2\}$ with any $f_1, f_2$.\\
By the dense of $\mathbb{R}$, there exists $e \in \mathbb{R}$ such that: $$ \sup\{\bar{L}(f_1) | f_1 \in V, f_1 \le g\} \le e \le \inf\{\bar{L}(f_2) | f_2 \in V, g \le f_2\}, $$
then we can extend $\bar{L}$ on $\bar{V} = V \oplus \mathbb{R}g$ with $\bar{L}(g) = e$ with $\bar{L}(f  + dg) = \bar{L}(f) + de, d \in \mathbb{R}$.\\
Case 1: $d = 0$, $\bar{L}(f + 0.g) = \bar{L}(f) \ge 0, \forall f \ge 0$.\\
Case 2: $d > 0$, $f + dg \ge 0 \Rightarrow f \ge -dg \Rightarrow -\dfrac{f}{d} \le g$\\
$\Rightarrow\bar{L}(-\dfrac{f}{d}) \le \bar{L}(g)$\\
$\Rightarrow  \bar{L}(\dfrac{f}{d}+g) \le 0 \Rightarrow \bar{L}(f +dg) \ge 0$.\\
Case 3: $d < 0$, $f + dg \ge 0 \Rightarrow f \ge -dg \Rightarrow -\dfrac{f}{d} \ge g \Rightarrow \bar{L}(-\dfrac{f}{d}) \ge \bar{L}(g)$\\
$\Rightarrow \bar{L}(\dfrac{f}{d}+g) \le 0 \Rightarrow \bar{L}(f + dg ) \ge 0$.\\
So in three cases, we have $(\bar{L}, V)$ is not maximum element, this is contradiction.\\
This implies $V = C'(X)$ or $L : \mathcal{A} \to \mathbb{R}$ can be extended on $C'(X)$ such that $\mathcal{A} \cup C_c(X) \subseteq C'(X)$ and preserves positivity. 
\end{proof}
As a consequence, we will prove one of the classical moment theorem, that is Hamburger's moment theorem. Actually, this theorem can be implied by the Haviland's theorem, but in here, we will give another proof by using Theorem \ref{thm1} and the Riesz Representation theorem. 
\begin{theorem}[Hamburger, 1920] The sequence of real numbers $\{s_n\}_n$ is the moment sequence, or there exists measure $\mu$ such that $$ s_k = \int_{-\infty}^{+\infty}x^k d\mu(x), k =0, 1, 2, \dots $$ if and only if $\sum\limits_{k,l=1}^N s_{k+l}c_k.{c_l} \ge 0$ for any finite choice $N\in \mathbb{N}$ and $c_1, \dots, c_N \in \mathbb{R}$.
\end{theorem}
\begin{proof}
If there exists measure $\mu$ such that $s_k = \int\limits_{-\infty}^{+\infty}x^k d\mu(x), k =0, 1, 2, \dots$, then we have $$\sum_{k,l=1}^N s_{k+l}c_k.{c_l} = \sum_{k,l=1}^N \int_{-\infty}^{+\infty}x^{k+l}d\mu\cdot c_k.{c_l} = \int_{-\infty}^{+\infty}(\sum_{k=1}^N c_k x^{k})^2d\mu \ge 0.$$

Conversely, suppose that $\sum\limits_{k,l=1}^N s_{k+l}c_k.\bar{c_l} \ge 0$. Considering functional $$L : \mathbb{R}[x] \to \mathbb{R}\ \text{with}\ L(x^k) = s_k; k = 0, 1, \dots,$$ then the sum $\sum\limits_{k,l=1}^N s_{k+l}c_k.\bar{c_l} \ge 0$ is equivalent $L[(\sum_{k=1}^N c_k x^{k})^2] \ge 0$ or $L[p^2(x)] \ge 0$ with $p(x) = \sum_{k=1}^N c_k x^{k}$.

If $f(x) \in \mathbb{R}[x]$ and $f(x) \ge 0, \forall x \in \mathbb{R}$ then by lemma \ref{lemma1} $$f(x) = p^2(x) + q^2(x); p(x), q(x) \in \mathbb{R}[x].$$ Hence, $L(f(x)) \ge 0$ with $f(x) \ge 0, \forall x \in \mathbb{R}$. Indeed, $$L(f(x)) = L(p^2(x) + q^2(x)) = L(p^2(x)) + L(q^2(x)) \ge 0.$$

Using Theorem \ref{thm1}, we get $\mathcal{A} = \mathbb{R}[x]$ can be extended on $C'(X)$ which contains $\mathbb{R}[x] \cup C_c(\mathbb{R})$. And from Riesz Representation Theorem, there exists measure $\mu$ such that $$L(g) = \int_\mathbb{R}g(x) d\mu(x), \forall g \in C_c(\mathbb{R}).$$ 
With a monomial $x^n$, there exists a sequence $g_{n,k} \in C_c(\mathbb{R})$ such that $$ g_{n,k}(x) = \begin{cases}
x^n & x \in [-k, k]\\
0 & x \notin [-(k+1), (k+1)]
\end{cases}$$
 satisfying $$ x^n = \lim\limits_{k \to +\infty} g_{n,k}(x).$$   

So this implies \begin{align*}
s_n = L(x^n) = L(\lim\limits_{k \to +\infty} g_{n,k}(x)) &= \lim\limits_{k \to +\infty}L(g_{n,k}(x))\\ 
&= \lim\limits_{k \to +\infty}\int_\mathbb{R}g_{n,k}(x) d\mu(x) = \int_\mathbb{R}\lim\limits_{k \to +\infty} g_{n,k}(x) d\mu(x)\\
 &= \int_\mathbb{R}x^n d\mu(x).
\end{align*}
The theorem is proved.
\end{proof}

\noindent\textbf{Acknowledgements}\\
The author wishes to thank Dr Ho Minh Toan, Department of Mathematical Analysis, Institute of Mathematics, Hanoi for his helpful discussions.\\
\bibliographystyle{amsalpha}

\begin{thebibliography}{AA}

\bibitem{Akh} 
N. I. Akhiezer, {\em The Classical Moment Problem and Some Related Questions
in Analysis}. English translation, Oliver and Boyd, Edinburgh, 1965.

\bibitem{Banach}
S. Banach, {\em Sur les fonctionelles lin\'{e}aires}, Studia Math. \textbf{1} (1929), 211-216.

\bibitem{BCR}
C. Berg, J. P. Christensen and P. Ressel, {\em Harmonic Analysis on Semigroups (Theory of Positive Definite and
Related Functions)}, Springer-Verlag, New York Inc., 1984.

\bibitem{Berg}
C. Berg, {\em The Multidimensional Moment Problem and Semigroups}, Proceedings of Symposia in Applied Mathematics, American Math. Scoc. (1987), Vol. \textbf{37}, 110-124. 

\bibitem{BM}
C. Berg and P. H. Maserick, {\em Polynomially positive definite sequences}, Math. Ann. \textbf{259} (1982), 487-495. 

\bibitem{Choq}
G. Choquet, {\em Lectures on analysis}, Volume II. Benjamin Math. Lecture Note Series, 1969.

\bibitem{Choq2}
G. Choquet, {\em Le probl\`eme des moments}, Séminaired 'Initiation \`{a} l'Analyse, Institut H. Poincaré, Paris, 1962.

\bibitem{Crist}
R. Cristescu, {\em Ordered Vector Spaces and Linear Operators}, Academiei, Bucharest and Abacus Press, Tunbridge Wells, Kent, 1976. 

\bibitem{Fuglede}
B. Fuglede, {\em The multidimensional moment  problem}, Expo. Math., no. \textbf{1} (1983), 47-65.

\bibitem{Ha1}
E. K. Haviland, {\em On the momentum problem for Distribution functions in more than one dimension I}, Amer. J. Math. \textbf{57} (1935), no. 3, 562-568.

\bibitem{Ha2}
E. K. Haviland, {\em On the momentum problem for Distribution functions in more than one dimension II}, Amer. J. Math. \textbf{58} (1936), no. 1, 164-168.

\bibitem{Hustad1}
O. Hustad, {\em On positive and continuous extension of positive functionals defined over dense subspaces}, Math. Scand. \textbf{7} (1959), 392-404.

\bibitem{Hustad2}
O. Hustad, {\em Linear inequalities and positive extension of linear functionals}, Math. Scand. \textbf{8} (1960), 333-338.

\bibitem{Kantorovich}
L. V. Kantorovich, {\em On the moment problem for a finite interval}, Dokl. Akad. Nauk., SSSR \textbf{14} (1937), 531–537 (Russian).

\bibitem{KS}
J.L. Kelley, T.P. Srinivasan, {\em Measure and Integral}, Volume I, Graduate Texts in Math. 116, Springer, 1988.

\bibitem{Krein}
M. Krein and V. Smulian, {\em On regularly convex sets in the space conjugate to a Banach space}, Ann. of Math. \textbf{41} (1940), 556-583.

\bibitem{Landau}
H. Landau, (ed.), {\em Moment in Mathematics}, Proceedings of Symposia in Applied Mathematics, vol. \textbf{37}, American Math. Scoc., 1987.

\bibitem{Lasserre}
J. B. Lasserre, {\em Global optimization with polynomials and the problem of moments}, Siam. J. Optim. vol. \textbf{11}, No. 3, 796-817.

\bibitem{Marshall1}
M. Marshall, {\em Approximating Positive Polynomials Using Sums of Squares}, Canad. Math. Bull. Vol. \textbf{46} (3) (2003), 400-418.

\bibitem{Marshall2}
M. Marshall, {\em Positive Polynomials and Sums of Squares}, Mathematical Surveys and Monographs, Vol. \textbf{46}, AMS, 2008.

\bibitem{MO}
S. Mazur and W. Orlicz, {\em Sur les espaces m\'{e}triques lin\'{e}aires II}, Studia Math. \textbf{13} (1953), 137-179.

\bibitem{Olteanu1}
O. Olteanu, {\em New Results on Markov Moment Problem}, International Journal of Analysis, vol. \textbf{2013}, 17p.

\bibitem{Olteanu2}
O. Olteanu, {\em Applications of Hahn-Banach principle to the moment problem}, Poincaré Journal of Analysis and Applications, \textbf{1} (2015), 1-28. 

\bibitem{FRiesz}
F. Riesz, {\em Sur une esp\`{e}ce de g\'{e}om\'{e}trie analytique des syst\`{e}mes de fonctions sommables}, C. R. Acad. Sci. Paris \textbf{149} (1907), 1409-1411.

\bibitem{MRiesz}
M. Riesz, {\em Sur le probl\`em des moment. Troisi\`{e}me Note}, Arkiv for matematik, astronomi ochfysik \textbf{17} (1923), no.16.

\bibitem{Schmudgen}
K. Schm\"{u}dgen, {\em The K-moment problem for compact semi-algebraic sets}, Math. Ann. \textbf{289} (1991), 203-206.

\bibitem{Sho} 
J. Shohat and J. D. Tamarkin, {\em The Problem of Moments}. Revised edition,
American Mathematical Society, Providence, 1950.

\bibitem{Si} 
B. Simon, {\em The classical moment problem as a self-adjoint finite difference
operator}. Adv. Math. {\bf 137} (1998), 82-203.


\end{thebibliography}

\end{document}